\documentclass{amsart}
\usepackage{amsmath,amssymb,amscd,ctable,latexsym,mathabx}
\usepackage{enumerate}
\usepackage[all]{xy}
\SelectTips{eu}{}

\CompileMatrices

\newcommand{\KRM}{{K^M}}

\newcommand{\ges}{{\geqslant}}
\newcommand{\les}{{\leqslant}}

\newcommand{\tope}{\operatorname{top}}

\newcommand{\card}{\operatorname{card}}
\newcommand{\chr}{\operatorname{char}}

\newcommand{\cent}[1]{{#1}^{\mathsf c}}
\newcommand{\col}{\colon}

\newcommand{\dd}{\partial}
\newcommand{\depth}{\operatorname{depth}}
\newcommand{\grade}{\operatorname{grade}}

\newcommand{\reg}{\operatorname{reg}}
\newcommand{\preg}[3]{{\operatorname{reg}}_{#1}^{#2}(#3)}

\newcommand{\hh}[1]{\operatorname{H}(#1)}
\newcommand{\HH}[2]{\operatorname{H}_{#1}(#2)}

\newcommand{\id}{\operatorname{id}}
\newcommand{\image}{\operatorname{Im}}
\newcommand{\Ker}{\operatorname{Ker}}
\newcommand{\Coker}{\operatorname{Coker}}

\newcommand{\pd}{\operatorname{pd}}
\newcommand{\rank}{\operatorname{rank}}

\newcommand{\Ext}[4]{\operatorname{Ext}^{#1}_{#2}(#3,#4){}}
\newcommand{\Hom}[3]{\operatorname{Hom}_{#1}(#2,#3)}
\newcommand{\Tor}[4]{\operatorname{Tor}_{#1}^{#2}(#3,#4){}}

\newcommand{\bsi}{{\boldsymbol{i}}}
\newcommand{\bsn}{{\boldsymbol{n}}}

\newcommand{\xra}{\xrightarrow}

\newcommand{\st}{$\ast$}
\newcommand{\bc}{$\scriptstyle\otriangleup$}
\newcommand{\et}{$\scriptstyle\triangledown$}
\newcommand{\ze}{$\scriptstyle\ovoid$}
\newcommand{\bu}{$\scriptstyle\textbf{0}$}
\newcommand{\ci}{$\scriptstyle\textbf{0}$}

\newcommand{\TTor}{\operatorname{Tor}}

\newcommand{\BN}{{\mathbb N}}

\newcommand{\BR}{{\mathbb R}}
\newcommand{\BZ}{{\mathbb Z}}

\theoremstyle{plain}

\newtheorem{theorem}{Theorem}[section]
\newtheorem{proposition}[theorem]{Proposition}
\newtheorem{lemma}[theorem]{Lemma}

\newtheorem{corollary}[theorem]{Corollary}

\theoremstyle{definition}

\newtheorem{example}[theorem]{Example}

\newtheorem{conjecture}[theorem]{Conjecture}
\newtheorem{examples}[theorem]{Examples}

\newenvironment{bfchunk}{\begin{chunk}\textbf}{\end{chunk}}
\newtheorem{chunk}[theorem]{}

\theoremstyle{remark}

\newtheorem*{Remark}{Remark}

\numberwithin{equation}{theorem}

\begin{document}

\title[Subadditivity of syzygies of Koszul algebras]
{Subadditivity of syzygies of Koszul algebras}

\author[L.~L.~Avramov]{Luchezar L.~Avramov}
\address{Department of Mathematics,
  University of Nebraska, Lincoln, NE 68588, U.S.A.}
\email{avramov@math.unl.edu}

\author[A.~Conca]{Aldo~Conca}
\address{Dipartimento di Matematica, 
Universit«a di Genova, Via Dodecaneso 35, 
I-16146 Genova, Italy}
\email{conca@dima.unige.it}

\author[S.~B.~Iyengar]{Srikanth B.~Iyengar}
\address{Department of Mathematics,
  University of Nebraska, Lincoln, NE 68588, U.S.A.}
\email{s.b.iyengar@unl.edu}
\date{\today}
\thanks{Part of this article is based on work supported by the National Science Foundation under Grant 
No.\,0932078000, while the authors were in residence at the Mathematical Sciences Research Institute 
in Berkeley, California, during the 2012--2013 Special Year in Commutative Algebra.  LLA was partly 
supported by NSF grant DMS-1103176 and a Simons Visiting Professorship; SBI was partly supported 
by NSF grant DMS-1201889 and a Simons Fellowship.}

  \begin{abstract}
Estimates are obtained for the degrees of minimal syzygies of quotient algebras of polynomial rings. 
For a class that includes Koszul algebras in almost all characteristics, these degrees are shown to 
increase by at most $2$ from one syzygy module to the next one.  Even slower growth is proved if, 
in addition, the algebra satisfies Green and Lazarsfeld's condition $N_q$ with $q\ge2$.
  \end{abstract}
  


\date{\today}
 
\maketitle

\section*{Introduction} 

In this paper we study homological properties of commutative graded algebras $R$, generated 
over a field $k$ by finitely many elements of degree one.  A common approach is to choose a presentation 
$R\cong S/J$, where $S$ is a standard graded polynomial ring over $k$ and $J$ an ideal containing no linear 
forms, and use the minimal free resolution $F$ of the graded $S$-module $R$.  In this context, the numbers
 \[
t_i^S(R)=\sup_{j\ges0}\{\TTor^S_i(R,k)_j\ne0\}
  \]
are of significant interest, as they bound the degrees of the basis elements of $F_i$.

Explicit upper bounds exist in terms of the most accessible data: $e$, the number of variables of $S$ and 
$t_1^S(R)$, the maximal degrees of the generators of $J$.  They are doubly exponential in $i$ and 
cannot be strengthened in general; see \ref{ch:exponential} for references. 

The situation is dramatically different when $R$ is Koszul; that is, when the minimal free resolution of $k=R/R_+$ 
over $R$ is linear; equivalently, when $t_{i}^{R}(k)\leq i$ for all $i \ge 0$. For Koszul algebras   Backelin \cite{Ba} 
and Kempf \cite{Ke} proved inequalities 
 \begin{equation}
 \label{eq:intro1}\tag{1}
t_i^S(R)\leq 2i
   \quad\text{for}\quad 1\le i\le \pd_SR\,.
 \end{equation}
Cases when equality hold are described in \cite{ACI}.

Here we extend and sharpen these results in several directions.

First, we relax the hypothesis, assuming only $t_{i}^{R}(k)\leq i$ for $i\le e+1$.  Such algebras are not necessarily 
Koszul and---unlike the latter---they are identified by a fixed finite segment of the minimal free resolution of $k$ 
over $R$; see~Roos \cite{Ro}. In Section~\ref{Decomposable} we prove that they satisfy \eqref{eq:intro1} and 
determine when equality holds.  

Second, for the algebras described above we investigate the existence of upper bounds on 
$t_i^S(R)$ that depend on $t_a^S(R)$ for $1\le a<i$.  In Section \ref{applications} we prove  
  \begin{equation}
 \label{eq:intro2}\tag{2}
 t_{i+1}^S(R)\leq t_{i}^S(R)+2 
   \quad\text{for}\quad 1\le i\le e-\dim R+1
  \end{equation}
when $(i+1)$ is invertible in $k$;  since then $t_1^S(R)=2$, this inequality refines \eqref{eq:intro1}
under additional hypotheses. It is itself the special case $b=1$ of the following inequality:
  \begin{equation}
  \label{eq:intro3}
t_{a+b}^S(R)\leq t_{a}^S(R)+t_{b}^S(R) 
    \quad\text{for}\quad a,b\ge1
    \text{ with }  a+b\le\pd_SR\,.
  \tag{3}
  \end{equation}
We conjecture that \eqref{eq:intro3} holds without restrictions on the characteristic of $k$ and
review evidence gleaned from a number of sources.  In particular, here we show that when 
$R$ is Cohen-Macaulay and $\binom{a+b}a$ is invertible in $k$ one has
  \[
 t_{a+b}^S(R)\leq t_{a}^S(R)+t_{b}^S(R)+1 
    \quad\text{for}\quad a\ge1
    \text{ and } b\ge2
    \text{ with } a+b\le\pd_SR\,.
  \]

For the algebras described above we study the impact of the Green--Lazarsfeld condition $N_q$:  
The differential $F_i\to F_{i-1}$ is given by a matrix of quadrics for $i=1$ and by matrices of linear 
forms for $2\le i\le q$; see \cite{GL}.  In Section \ref{applications2} we show that (excluding finitely 
many specified characteristics) it implies a sharpening of \eqref{eq:intro1}:
 \begin{equation}
   \label{eq:intro5}\tag{4}
t_i^S(R)\leq 
2\left\lfloor \frac{i}{q+1}\right\rfloor+i+
\begin{cases}
0 & \text{if}\quad (q+1)|i
  \\
1& \text{otherwise. }
\end{cases}
  \end{equation}
This inequality reveals hitherto unknown properties of the resolutions of algebras of geometric interest,  
even in cases such as Segre products and Veronese subalgebras of polynomial rings. For related recent 
progress in this context see \cite{Ru, Sn}.

Our approach utilizes the homology of the Koszul complex $K$ on a basis of~$R_1$.  

The differential bigraded algebra $K$ provides a bridge between the homology of the $R$-module $k$ 
and that of the $S$-module $R$; see \cite{Av:barca}.  As in \cite{ACI}, the proof of \eqref{eq:intro1} 
depends on the decomposable classes in $\hh K$ contained in its subalgebra generated by $\HH1K$.  
On the other hand, the proofs of \eqref{eq:intro2} and \eqref{eq:intro5} involve an analysis of 
indecomposable homology classes; namely, of the quotient $\HH{a+b}K/\HH aK\cdot\HH bK$.  

This is the main technical innovation in the paper.  Its effectiveness depends on the existence (in almost 
all characteristics) of splitting maps for Koszul cycles, discovered in \cite{BCR2} and described in Section 
\ref{KoszulComplexes} from a more conceptual perspective. These maps are combined with cascades
of bounds on regularities of Tor modules, obtained in Sections \ref{Flat} and \ref{mainthm}.  The material 
in the first three sections is presented in its natural generality, which is much greater than the one needed 
in the sequel.
  
\section{Complexes of flat modules} 
   \label{Flat}
   
In this section we evaluate how cycles change with the introduction of coefficients, starting from
a general situation and moving on to a graded one.  Only basic homological 
considerations are involved.

  \begin{bfchunk}{Coefficients.}
    \label{ch:phi}
Let $R$ be an associative ring and $\cent R$ its center.  Let 
  \begin{equation}
       \label{eq:F}
F=\quad \cdots\to F_{a+1}\xra{\,\dd_{a+1}\,} F_a\xra{\,\dd_a\,} F_{a-1}\to \cdots
  \end{equation}
be a complex of right $R$-modules, and for every integer $a$ set 
  \[
Z_a=\Ker(\dd_a)\,,\quad
B_a=\image(\dd_{a+1})\,,\quad
H_a=Z_a/B_a\,,\quad
\text{and}\quad
C_{a}=\Coker(\dd_{a+1})\,.
  \]
  
For each left $R$-module $N$ and for $Z_a(F\otimes_R N)=\Ker(\dd_a\otimes N)$ the assignment 
$z\otimes x\mapsto z\otimes x$ defines a natural in $N$ homomorphism of abelian groups
  \begin{equation}
       \label{eq:phi}
\phi_{a}\col Z_a\otimes_RN \to Z_a(F\otimes_R N).
  \end{equation}
  \end{bfchunk}

The next lemma is a sort of ``universal coefficients theorem'' for cycles.

\begin{lemma}
\label{lem:serra}
When $F_a$ and $F_{a-1}$ are flat there is a natural in $N$ exact sequence
  \[
0\to  \TTor_1^R(B_{a-1},N)\to Z_a\otimes_R N\xra{\,\phi_a\,} Z_a(F\otimes_R N)\to \TTor_1^R(C_{a-1},N) \to 0
  \]   
of $\cent R$-modules, and for every $i\ge 1$ there is a natural in $N$ isomorphism
  \begin{align}
  \label{eq:serra0}
\TTor_{i+1}^R(B_{a-1},N)&\cong\TTor_{i}^R(Z_a,N)\,.
  \end{align}
 \end{lemma} 

 \begin{proof} 
The maps $F_a\xra{\pi_{a}}B_{a-1}$ and $B_{a-1}\xra{\iota_{a-1}}F_{a-1}$  yield a commutative diagram 
  \begin{equation*}
\xymatrixcolsep{1pc}
\xymatrixrowsep{2pc} 
\xymatrix{
0 
\ar@{->}[r]
&\Ker(\pi_{a}\otimes_RN)
\ar@{->}[r]
\ar@{->}[d]
& F_a\otimes_R N
\ar@{->}[rr]^-{\pi_a\otimes_RN}
\ar@{->}[d]^-{\dd_a\otimes_RN}
&& B_{a-1}\otimes_R N
\ar@{->}[r]
\ar@{->}[d]^-{\iota_{a-1}\otimes_RN}
&0
  \\
0 
\ar@{->}[r]
&0 
\ar@{->}[r]
& F_{a-1}\otimes_R N
\ar@{=}[rr]
&& F_{a-1}\otimes_R N
}
     \end{equation*}
with exact rows.  The Snake Lemma then produces an exact sequence
  \[
0\to\Ker(\pi_{a}\otimes_RN)\to Z_a(F\otimes_R N) \to \Ker(\iota_{a-1}\otimes_RN) \to 0
  \]
As $F_a$ is flat, the exact sequence $0\to Z_a\to F_a\xra{\pi_a} B_{a-1}\to 0$ induces the isomorphisms 
\eqref{eq:serra0} and yields an exact sequence 
  \begin{align}
  \notag
0\to  \TTor_1^R(B_{a-1},N)&\to 
Z_a\otimes_R N\xra{\,\zeta_a\,} \Ker(\pi_{a}\otimes_RN)\to0
  \end{align}
such that the composition of $\zeta_a$ with the inclusion $Z_a(F\otimes_R N)\subseteq F_a\otimes_R N$
is equal to $\phi_a$.  As $F_{a-1}$ is flat we also have exact sequence
  \[
0\to \TTor^R_1(C_{a-1},N)  \to B_{a-1}\otimes_R N\xra{\iota_{a-1}\otimes_RN} F_{a-1}  \otimes_R N 
  \]
induced by $0\to B_{a-1}\xra{\iota_{a-1}} F_{a-1} \to C_{a-1}\to 0$; it computes $\Ker(\iota_{a-1}\otimes_RN)$.
 \end{proof}

Next we describe notation and conventions concerning gradings in this paper.

   \begin{bfchunk}{Graded objects.}
     \label{ch:graded}
Let $k$ be a field and $V=\bigoplus_jV_j$ a graded vector space; set 
  \[
\tope(V)=\sup\{ j \mid V_j\neq 0\}\,;
  \]
thus, $\tope(V)=-\infty$ if and only if $V=0$.  The notation $\deg(v)=j$ means $v\in V_j$.

A \emph{graded $k$-algebra} is a graded vector space $R=\bigoplus_{i\in\BZ}R_i$ with 
$R_0=k$, $R_i=0$ for $i<0$, and associative $k$-bilinear products $R_i\times R_j\to R_{i+j}$. 
A \emph{left graded $R$-module} is a graded vector space $N=\bigoplus_{j\in\BZ}N_j$ with 
$N_j=0$ for $j\ll0$ and associative $k$-bilinear products $R_i\times N_j\to N_{i+j}$.  
\emph{Graded right $R$-modules} are defined similarly.  

\emph{Homomorphisms} of graded $R$-modules are homogeneous of degree~$0$. Thus, 
when $L$ is a right $R$-module and $N$ a left one, $\TTor_i^R(L,N)$ is a graded $\cent R$-module.

As usual, given an graded $R$-module $N$ and integer $n$, we write $N(n)$ for the graded $R$-module with $N(n)_{i}=N_{n+i}$ for all $i\in\BZ$.

By abuse of notation, $k$ also stands for the module $R/\bigoplus_{i\ges1}R_i$.  The number
  \[
t_i^R(N)=\tope(\TTor_i^R(k,N))
  \]
then equals the largest degree of a minimal $i$-syzygy of the $R$-module $N$.
   \end{bfchunk}
  
  \begin{lemma} 
\label{elem1}
When $L$ is a right graded $R$-module and $N$ a left one, one has
  \begin{equation}
  \label{eq:elem1}
\tope(\TTor_i^R(L,N))\leq \tope(L)+t_i^R(N)
\quad\text{for all}\quad
i\in\BZ\,.
  \end{equation}
\end{lemma}

\begin{proof}  
Set $t=\tope(L)$. We may assume $t$ is finite, and then $s=t-\min\{j \mid L_j\neq 0\}$ is a 
non-negative integer.  If $s=0$, then $\TTor_i^R(L,N)\cong L\otimes_k\TTor_i^R(k,N)$, 
so equality holds in \eqref{eq:elem1}.  The exact sequence $0\to L_t(-t)\to L\to L/L_t(-t)\to0$
of graded $R$-modules induces an exact sequence of graded $\operatorname{Tor}$ vector 
spaces, from which the inequality \eqref{eq:elem1} follows by induction on $s$.
\end{proof} 

\begin{lemma} 
\label{elem2}
Let $R$ be a graded algebra and $F$ a complex of flat graded $R$-modules. For all integers $i\ge1$ and $a\geq 0$ the following inequalities hold:
  \begin{align}
    \label{eq:elem2.B}
\tope(\TTor_i^R(B_a,N))&\leq \max_{0\les j\les a}\{ \tope(H_j)+t^R_{a+i+1-j}(N)\}
  \\
    \label{eq:elem2.Z}
\tope(\TTor_i^R(Z_a,N))&\leq \max_{0\les j< a}\{ \tope(H_j)+t^R_{a+i+1-j}(N)\}
  \\
    \label{eq:elem2.C}
\tope(\TTor_i^R(C_{a-1},N))&\leq \max_{0\les j< a}\{ \tope(H_j)+t^R_{a-1+i-j}(N)\}
  \end{align}
  \end{lemma}

\begin{proof}  
Fix some integer $i\ge1$. In view of \eqref{eq:serra0}, we obtain an exact sequence
  \begin{equation}
  \label{eq:elem2}
\TTor_{i+1}^R(H_a,N) \to  \TTor_i^R(B_a,N) \to   \TTor_{i+1}^R(B_{a-1},N)
  \end{equation}
from the one induced by $0\to B_a\to Z_a\to H_a\to 0$. Now \eqref{eq:elem2} and \eqref{eq:elem1} yield
  \[
\tope(\TTor_i^R(B_a,N))\leq \max\{ \tope(H_a)+t_{i+1}^R(N), \tope(\TTor_{i+1}^R(B_{a-1},N))\}
  \]
The proof of \eqref{eq:elem2.B} is completed by iterating this procedure. 

The isomorphisms \eqref{eq:serra0} show that formula \eqref{eq:elem2.Z} follows from \eqref{eq:elem2.B}.

The exact sequence $0\to H_{a-1}\to C_{a-1}\to B_{a-2}\to 0$ induces an exact sequence
  \[
\TTor_{i}^R(H_{a-1},N) \to  \TTor_i^R(C_{a-1},N) \to   \TTor_i^R(B_{a-2},N)
  \]
from which the inequality \eqref{eq:elem2.C} follows, due to formulas  \eqref{eq:elem1} and \eqref{eq:elem2.B}. 
  \end{proof} 
 
\section{Indecomposable Koszul homology}
\label{KoszulComplexes}

In this section $R$ denotes a commutative ring and $M$ an $R$-module.

\begin{bfchunk}{Koszul complexes.}
  \label{ch:koszul}
Let $E=\bigwedge_{R}E_1$ be the exterior algebra on a free $R$-module $E_1$ sitting 
in homological degree $1$, and $\mu\col E\otimes_{R}E\to E$ the product map.

The algebra $E$ has following universal property:  Each $R$-linear map $E_1\to R$  
extends uniquely to an $R$-linear map $\dd\col E\to E $ of homological degree $-1$, satisfying
\begin{equation}
\label{eq:delta}
\dd\mu=\mu(\dd\otimes E+E\otimes\dd)\,.
\end{equation}
As $\dd^2(E_1)=0$ holds for degree reasons, this implies $\dd^2=0$.  The complex $K$ with underlying 
graded module $E$ and differential~$\dd$ is the \emph{Koszul complex} of $E_1\to R$.

Set $\KRM=K\otimes_RM$. In view of \eqref{eq:delta}, $K$ is a DG algebra and $K^M$ is a DG $K$-module.
Thus, $H(K)$ is a graded algebra and $H(K^M)$ is a graded $H(K)$-module.

If $R$ is a graded $k$-algebra and $E_1\to R$ a homomorphism of graded $R$-modules, see \ref{ch:graded}, 
then the differential algebra $K$ and its differential module $K^M$ are \emph{bigraded}: their differentials 
decrease homological degrees by $1$ and preserve internal degrees.  Thus, $\hh R$ is naturally a bigraded 
algebra and $\hh M$ a bigraded module over it  \end{bfchunk}

\begin{theorem}
\label{thm:epi}
For each pair $(a,b)\in\BZ^2$ there exists a natural $R$-linear map
  \begin{equation}
    \label{eq:epi}
\gamma_{a,b}\col \TTor_1^R(C_{a-1}(K),Z_b(K^M))
\to\frac{H_{a+b}(K^M)}{H_a(K)H_b(K^M)}
  \end{equation}
described in \emph{\ref{ch:prod}}, which is surjective when $\binom{a+b}{a}$ is invertible in $R$.

When $R$ and $M$ are graded $\gamma_{a,b}$ is a homomorphism of graded modules.
 \end{theorem} 

The theorem is proved at the end of the section. 

\begin{corollary}
\label{cor:epi}
Assume $\HH aK=0$ for $a\ge1$ (for instance, $E_1$ has a basis $e_1,\dots,e_r$ such that the 
sequence $\dd(e_1),\dots,\dd(e_r)$ is $R$-regular) and set $I=\dd(E_1)$.

If $\binom{a+b}{a}$ is invertible in $R$, then there is a surjective $R$-linear map
 \begin{equation*}
\gamma'_{a,b}\col \TTor_a^R(R/I,Z_b(K^M))\to\TTor_{a+b}^R(R/I,M)
  \end{equation*}
 \end{corollary} 
 
 \begin{proof}
The hypothesis yields $H_{a+b}(K^M)\cong\TTor_{a+b}^R(R/I,M)$ and an exact sequence
  \[
0\to C_{a-1}(K)\to K_{a-1}\to\cdots\to K_0\to R/I\to0
  \]
 of $R$-modules.  Since $K_i$ is flat for $0\le i\le a-1$, the iterated connecting map 
 $\TTor_a^R(R/I,Z_b(K^M))\to\TTor_1^R(C_{a-1}(K),Z_b(K^M))$ is bijective.
 \end{proof}

As a special case, we obtain an (unexpected) inequality between Betti numbers: 

 \begin{corollary}
\label{specor:epi}
Let  $R$ be a regular local ring, $K$ the Koszul complex on a minimal generating set of the 
maximal ideal of $R$, and $M$ a finitely generated $R$-module.

If $\binom{a+b}{a}$ is invertible in $R$, then the $b$-cycles of  $K^M=K\otimes_RM$ satisfy   
  \[
\beta^R_a(Z_b(K^M))\geq\beta^R_{a+b}(M)\,.
  \]
In the graded setup such inequalities hold for graded Betti numbers. 
 \qed
   \end{corollary} 
  
\begin{bfchunk}{Diagonal maps.}
  \label{ch:diagonal}
Equipping $E\otimes_{R}E$ with the product 
\[
(x_{1} \otimes x_{2}) \cdot (y_{1} \otimes y_{2}) = (-1)^{|x_{2}||y_{1}|}(x_{1}y_{1}\otimes x_{2}y_{2})\,.
\]
turns it into an $R$-algebra that is strictly graded-commutative for the homological degree.  The universal property of exterior algebras 
yields a unique homomorphism $\Delta\col E\to E\otimes_{R}E$ of graded $R$-algebras, such that 
$\Delta(x)=x\otimes1+1\otimes x$ for $x\in E_1$.

Set $\dd'=\dd\otimes E$ and $\dd''=E\otimes \dd$.   These maps satisfy the equalities
  \begin{gather}
    \label{eq:dd1}
\dd'\Delta=\Delta\dd=\dd''\Delta\\
     \label{eq:dd3}
\dd'\dd''+\dd''\dd'=0
 \end{gather}
Indeed, the maps in \eqref{eq:dd1} and \eqref{eq:dd3} are graded derivations of the $R$-algebras $E$ 
and $E\otimes_RE$, respectively.  They are generated in homological degree $1$, so it suffices to verify 
agreement on $E_{1}$ and $(E\otimes_RE)_1$, respectively.  This is straightforward.
  \end{bfchunk}

\begin{bfchunk}{Partial cycles.}
  \label{ch:partial}
As $E$ is a graded free $R$-module, we have
  \[
E\otimes_RZ(\KRM)=E\otimes_R\Ker(\dd\otimes M)=\Ker(\dd''\otimes M)
  \]
From \eqref{eq:dd3} we obtain $(\dd'\otimes M)(\Ker(\dd''\otimes M))\subseteq\Ker(\dd''\otimes M)$, so 
$\dd'\otimes M$ turns $\Ker(\dd''\otimes M)$ into a complex.  It is equal to $K\otimes_RZ(\KRM)$, so we get
  \begin{equation}
  \begin{aligned}
      \label{eq:dd4}
Z(K\otimes_RZ(\KRM))&=\Ker(\dd'\otimes M)\cap\Ker(\dd''\otimes M)\\
  \end{aligned}
  \end{equation}
as graded submodules of $E\otimes_RE\otimes_RM$.  
For every integer $n$ there is an equality
\begin{equation}
\label{eq:dd5}
Z_{n}(K\otimes_R Z(\KRM)) = \bigoplus_{i+j=n} Z_{i}(K\otimes_RZ_j(\KRM))
\end{equation}
Let $\iota_{i,j}$ and $\pi_{i,j}$ denote the canonical maps from and to $Z_{i}(K\otimes_RZ_j(\KRM))$, respectively,
induced by the decomposition \eqref{eq:dd5}.
  \end{bfchunk}

\begin{bfchunk}{The maps $\alpha$.}
  \label{ch:alpha}
The equalities \eqref{eq:delta} and \eqref{eq:dd3} imply
  \[
(\mu\otimes M)(\Ker(\dd'\otimes M +\dd''\otimes M))\subseteq\Ker(\dd\otimes M)=Z(\KRM)\,.
  \]
In view of \eqref{eq:dd4}, $\mu\otimes M$ restricts to a homomorphism of graded $R$-modules
  \[
\alpha\col Z(K\otimes_RZ(\KRM))\to Z(\KRM)\,.
  \]

Setting $n=a+b$, we define $\alpha_{a,b}$ to be the composed map
  \[
Z_{a}(K\otimes_RZ_b(\KRM))\xra{\iota_{a,b}}Z_{n}(K\otimes_RZ(\KRM))
\xra{\alpha_{n}}Z_{n}(\KRM)
  \]
 \end{bfchunk}

\begin{bfchunk}{The maps $\beta$.} 
The equalities \eqref{eq:dd1} imply  
\[
(\Delta\otimes_{R}M)(\Ker(\dd\otimes M))\subseteq \Ker(\dd'\otimes M)\cap\Ker(\dd''\otimes M)\,.
\]
In view of \eqref{eq:dd4}, $\Delta\otimes_{R}M$ restricts to a homomorphism of graded $R$-modules
 \[
\beta\col Z(K\otimes M)\to Z(K\otimes_RZ(\KRM))\,.
\]

Setting $n=a+b$, we define $\beta_{a,b}$ to be the composed map
  \[
Z_{n}(\KRM)\xra{\beta_{n}}Z_{n}(K\otimes_RZ(\KRM))
\xra{\pi_{a,b}}Z_{a}(K\otimes_RZ_b(\KRM))
  \]
  \end{bfchunk}

The next result is proved by Bruns, Conca, and R\"omer in \cite[2.4]{BCR2} through 
direct calculations.  Our proof mines the graded Hopf algebra structure of $E$.

\begin{lemma}
\label{lem:BRS}
For each pair $(a,b)$ of non-negative integers there is an equality
   \[
\alpha_{a,b}\circ\beta_{a,b}=\binom{a+b}{a}\id^{Z_{a+b}(\KRM)}\,.
  \]
  \end{lemma}

\begin{proof}
Set $n=a+b$ and note that $\alpha_{a,b}\circ\beta_{a,b}$ is induced by the composed map
\begin{equation*}
\delta_{a,b}\col 
E_{n} \xra{\Delta_{n}} 
(E\otimes_{R}E)_{n} \xra{\rho_{a,b}} 
(E\otimes_{R}E)_{n} \xra{\mu_{n}} 
E_{n}
\end{equation*}
where $\rho_{a,b}$ projects $(E\otimes_RE)_n=\bigoplus_{i+j=n}E_{i}\otimes_R E_{j}$ onto its summand $E_{a}\otimes E_{b}$.

Let $n$ be a positive integer and let $\bsn$ denote the sequence $(1,\dots,n)$. For each sequence $\bsi=(u_{1},\dots,u_{i})$ of integers satisfying $1\le u_{1}< \cdots <u_{i}\le n$, write $\bsn\smallsetminus\!\bsi= (v_{1},\dots,v_{n-i})$ with $v_{1}< \cdots <v_{n-i}$ and let $\mathrm{sgn}(\bsi,\bsn\smallsetminus\!\bsi)$ denote the sign of the permutation $(u_{1},\dots,u_{i},v_{1},\dots,v_{n-i})$ of $\bsn$. Given a subset $x=\{x_1, \dots, x_{n}\}$ of $E_1$, setting $x_\bsi = x_{u_1}\wedge \cdots \wedge x_{u_i}\in E_{i}$ we obtain the following equalities:
  \begin{align*}
\delta_{a,b}(x_{\bsn}) 
&= \mu_{n}\circ\rho_{a,b}\bigg(\sum_{\bsi\subseteq\bsn}\mathrm{sgn}(\bsi,\bsn\smallsetminus\!\bsi)\, x_{\bsi}\otimes x_{\bsn\smallsetminus \bsi}\bigg) \\
&= \mu_{n}\bigg(\sum_{\card{\bsi}=a}\mathrm{sgn}(\bsi,\bsn\smallsetminus\!\bsi)\, x_{\bsi}\otimes x_{\bsn\smallsetminus \bsi}\bigg) \\
&= \binom{n}a x_{\bsn}
  \end{align*}
It remains to remark that $E_n$ is additively generated by elements of the form $x_\bsn$.
   \end{proof}

   \begin{bfchunk}{The maps $\gamma$.}
     \label{ch:prod}
Each pair $(a,b)$ of integers defines a diagram of $R$-linear maps
  \begin{equation}
    \label{eq:prod}
      \begin{gathered}
\xymatrixcolsep{1pc}
\xymatrixrowsep{2pc} 
\xymatrix{
& Z_a(K)\otimes Z_b(K^M)
\ar@{->}[rr]^{\chi_{a}\otimes\chi_{b}}
\ar@{->}[dl]_-{\phi_{a}}
\ar@{->}[d]^-{\mu^{Z}_{a,b}}
&& H_a(K)\otimes H_b(K^M)
\ar@{->}[d]^-{\mu^{H}_{a,b}}
  \\
Z_{a}(K\otimes_R Z_b(K^M))
\ar@{->}[r]^-{\alpha_{a,b}}
& Z_{a+b}(K^M)
\ar@{->}[rr]^{\chi_{a+b}}
&& H_{a+b}(K^M)
}
   \end{gathered}
     \end{equation}
where the maps $\chi$ are canonical, $\mu^{Z}_{a,b}$ and $\mu^{H}_{a,b}$ are induced by $\mu$ 
in view of \eqref{eq:delta}, $\alpha_{a,b}$ comes from \ref{ch:alpha}, and $\phi_{a}$ is given by
\ref{ch:phi} with $F=K$ and $N=Z_b(K^M)$.  The definitions show that diagram \eqref{eq:prod}  
commutes, so it yields natural $R$-linear maps 
  \begin{equation}
    \label{eq:epi2}
\Coker(\phi_{a})\to\Coker(\mu^{Z}_{a,b})\to\Coker(\mu^{H}_{a,b})
  \end{equation}
The isomorphism $\TTor_1^R(C_{a-1}(K),Z_b(K^M))\cong\Coker(\phi_{a})$ from Lemma \ref{lem:serra}, 
composed with the maps in \eqref{eq:epi2}, defines the map $\gamma_{a,b}$ in \eqref{eq:epi}.
    \end{bfchunk}

   \begin{proof}[Proof of Theorem \emph{\ref{thm:epi}}]
Since $\chi_{a+b}$ is surjective, so is the second map in \eqref{eq:epi2}.

If $\binom{a+b}{b}$ is invertible in $R$, then $\alpha_{a,b}$ is surjective by Lemma \ref{lem:BRS},
hence so is the first map in \eqref{eq:epi}.  It follows that $\gamma_{a,b}$ is surjective, as desired.

When $R$ is a graded ring and $M$ is a graded $R$-module the constructions of 
 $\alpha_{a,b}$, $\beta_{a,b}$, and $\gamma_{a,b}$ show that they are homomorphisms of graded $R$-modules.
  \end{proof}

 \section{Regularity} 
 \label{mainthm}
 
In this section $k$ is a field and $R$ a graded commutative $k$-algebra with $R_0=k$, $R=k[R_1]$,
and $\rank_kR_1$ finite.  We write $S$ for the symmetric $k$-algebra $S$ on $R_1$ and let $S\to R$ 
be the canonical surjective homomorphism of graded $k$-algebras. 

  \begin{bfchunk}{Regularities.}
    \label{ch:regs}
The $n$th \emph{partial regularity} of an $R$-module $M$ is the number
   \begin{equation}
    \label{eq:regs}
\preg{n}RM=\max_{i\les n}\{t^R_i(M)-i\}\,.
   \end{equation}
Thus, for every $n\in\BZ$ there is an inequality
  \begin{alignat}{2}
    \label{eq:m(R)2}
\preg{n}{}R&\leq\preg{n+1}{}R
  \end{alignat}
and $\sup\{\preg{n}RM\}_{n\in\BZ}$ is the \emph{Castelnuovo-Mumford regularity} $\reg^R(M)$.
  \end{bfchunk}

The goal of this section is to prove the following result.
 
 \begin{theorem} 
 \label{subad}
Let $a$, $b$ be non-negative integers satisfying $a+b\le\pd_SR$.

If $\binom{a+b}{b}$ is invertible in $k$, then for every $R$-module $M$ one has
   \[
t_{a+b}^S(M)\leq 
\max\left\{ 
\begin{aligned}
&t^S_a(R)+ t^S_b(M)
  \\ 
&\preg{a-1}SR+\preg{a+b}RM+a+b
  \\
&\preg{a-1}SR+\preg{b-1}SM+\preg{a+b+1}Rk +a+b+1
\end{aligned}
\right\}
  \]
 \end{theorem} 

We single out a special case of the theorem:

 \begin{corollary} 
  \label{cor:subad}
   \pushQED{\qed} 
If $\preg{a+b+1}Rk=0=\preg{a+b}RM$, then there is an inequality 
  \begin{equation}
    \label{eq:subad}
t_{a+b}^S(M)\leq 
\max\left\{ 
\begin{aligned}
&t_a^S(R)+ t_b^S(M)
  \\
&\preg{a-1}SR+\preg{b-1}SM+a+b+1
\end{aligned}
\right\}
\qedhere
  \end{equation}
  \end{corollary} 

The proof of Theorem \ref{subad} utilizes Koszul complexes.

   \begin{bfchunk}{Koszul homology.}
     \label{ch:koszulHomology}
The inclusion $R_1(-1)\subseteq R$ defines an $R$-linear map $R\otimes_kR_1(-1)\to R$, and
hence a Koszul complex $K$; see \ref{ch:koszul}.  Set $K^M=K\otimes_RM$.

For each $i\in\BN$ we have an isomorphism 
  \begin{equation}
    \label{eq:koszul0}
H_i(K^M)\cong\TTor^S_i(k,M)
     \end{equation}
of graded $R$-modules; it gives an equality
  \begin{equation}
    \label{eq:koszul1}
t^S_i(M)=\tope(H_i(K^M))
  \end{equation}
As $(R_+)H(K^M)=0$, for all non-negative integers $i,j$ the preceding formula yields
  \begin{equation}
    \label{eq:koszul2}
t^R_{j}(H_{i}(K^M))=t^S_{i}(M)+t^R_{j}(k)
  \end{equation}
    \end{bfchunk}

By applying Theorem \ref{thm:epi}, then formula \eqref{eq:elem2.C} with $F=K$ and $N=Z_b(K^M)$, 
and finally formula \eqref{eq:koszul1} with $M=R$ we obtain:

\begin{lemma}
  \label{serra}
    \pushQED{\qed} 
If $\binom{a+b}{b}$ is invertible in $R$, then there is an inequality
  \begin{equation}
    \label{eq:serra}
\tope\left(\frac{H_{a+b}(K^M)}{H_a(K)H_b(K^M)}\right)
\leq\max_{0\les j<a}\{ t^S_j(R)+t^R_{a-j}(Z_b(K^M)) \}\,.
   \qedhere
  \end{equation}
 \end{lemma} 
 
Next we estimate the partial regularities of the cycles appearing in \eqref{eq:serra}.

   \begin{lemma} 
 \label{banale} 
 For each $R$-module $M$ the following inequality holds: 
 \begin{equation}
   \label{eq:banale}
t^R_a(Z_b(K^M))
 \leq 
 \max_{1\les j\les b}
 \left\{
\begin{aligned}
&t^R_a(M)+b
  \\
&t_{a+j}^R(M)+b-j
  \\
&t_{b-j}^S(M)+t^R_{a+j+1}(k)
 \end{aligned}
 \right\}
 \end{equation}
 \end{lemma}
 
 \begin{proof}   
From the exact sequences of $R$-modules 
  \[
\xymatrixcolsep{.8pc}
\xymatrixrowsep{.1pc} 
\xymatrix{
0\ar@{->}[r]
&Z_b(K^M)\ar@{->}[r]
&K_b^M\ar@{->}[r]
&B_{b-1}(K^M)\ar@{->}[r]
&0
  \\
0\ar@{->}[r]
&B_{b-1}(K^M)\ar@{->}[r]
&Z_{b-1}(K^M)\ar@{->}[r]
&H_{b-1}(K^M)\ar@{->}[r]
&0
}
  \]
we obtain the following inequalities:
\begin{align*}
t^R_a(Z_b(K^M))
&\leq \max\{t^R_{a+1}(B_{b-1}(K^M),t^R_a(M)+b)\} 
\\ 
t^R_{a+1}(B_{b-1}(K^M))
&\leq   \max\{t^R_{a+2}(H_{b-1}(K^M)), t^R_{a+1}(Z_{b-1}(K^M)) \}.
\end{align*}
By concatenating these relations and invoking \eqref{eq:koszul2} we get
 \[
t^R_a(Z_b(K^M))
 \leq 
 \max\left\{
\begin{aligned}
&t^R_a(M)+b
  \\
&t^R_{a+1}(Z_{b-1}(K^M))
  \\
&t_{b-1}^S(M)+t^R_{a+2}(k)
  \end{aligned}
 \right\}
   \]

The desired result follows from the preceding formula by induction on $b$.  
 \end{proof}
 
  \begin{proof}[Proof of Theorem \emph{\ref{subad}}]
Products in homology yield an exact sequence 
  \[
H_a(K)\otimes H_b(K^M)\to H_{a+b}(K^M)\to\frac{H_{a+b}(K^M)}{H_a(K)H_b(K^M)}\to0
  \]
of graded $k$-vector spaces.  It accounts for the first inequality in the following display, where the equality comes from \eqref{eq:koszul1}, 
the second inequality from \eqref{eq:serra}, and the last inequality from \eqref{eq:banale}:
  \begin{align*}
t^S_{a+b}(M)
&=\tope(H_{a+b}(K^M))\\
&\leq \max\left\{\tope(H_{a}(K))+\tope(H_{b}(K^M)), \tope\left(\frac{H_{a+b}(K^M)}{H_a(K)H_b(K^M)}\right) \right\}
  \\
&\leq\max_{0\les i<a}\{t^S_a(R)+ t^S_b(M), t^S_i(R)+t^R_{a-i}(Z_b(K^M))\}
  \\
&\leq \max_{\substack{0\les i<a\\ 1\les j\les b}}
\left\{
\begin{aligned}
&t^S_a(R)+ t^S_b(M)
  \\
&t^S_i(R)+t^R_{a-i}(M)+b
  \\
&t^S_i(R)+t_{a-i+j}^R(M)+b-j 
  \\
&t^S_i(R)+t_{b-j}^S(M)+t^R_{a-i+j+1}(k)
 \end{aligned}
 \right\}
   \end{align*}
For each pair $(i,j)$ with $0\le i<a$ and $0\le j\le b$, in view of \eqref{eq:regs} we have
\begin{align*}
t^S_i(R)&+t_{a-i+j}^R(M)+b-j\\
&\leq [\preg{a-1}SR+i]+[\preg{a+b}RM+a-i+j]+b-j\\
&=\preg{a-1}SR+\preg{a+b}RM+a+b
\end{align*}
Similarly, for each pair $(i,j)$ with $0\le i<a$ and $1\le j\le b$, we obtain
  \begin{align*}
t^S_i(R)&+t_{b-j}^S(M)+t^R_{a-i+j+1}(k)\\
 &\leq [\preg{a-1}SR+i]+[\preg{b-1}SM+b-j]+[\preg{a+b+1}Rk+a-i+j+1]\\
  &= \preg{a-1}SR+\preg{b-1}SM+\preg{a+b+1}Rk+a+b+1
\end{align*}

It remains to assemble the inequalities in the last three displays.
  \end{proof} 
 
 \section{Decomposable Koszul homology}
\label{Decomposable}
We keep the hypotheses on $R$ and $S$ from Section \ref{mainthm}, and let $J$ denote the kernel 
of the surjection $S\to R$ of graded $k$-algebras.

\begin{theorem} 
  \label{thm:aci1}
Let $K$ be the Koszul complex described in \emph{\ref{ch:koszulHomology}}. 

If $\preg{n+1}Rk=0$ for some $n\ge0$, then for $0\le i\le n$ there are equalities
  \begin{align}
    \label{eq:aci11}
\HH iK_j&=0 \quad\text{when}\quad j>2i\,.
  \\
    \label{eq:aci12}
\HH iK_{2i}&=(\HH 1K_{2})^i\,.
  \end{align}
   \end{theorem} 

Formula \eqref{eq:koszul1} gives a numerical translation of the equality \eqref{eq:aci11}.

\begin{corollary} 
  \label{cor:protoKoszul}
   \pushQED{\qed} 
For $0\le i\le n$ the following inequalities hold:
  \begin{equation}
    \label{eq:kb}
t^S_i(R)\le 2i \,.
   \qedhere
  \end{equation}
   \end{corollary} 
 
When $R$ is Koszul, that is to say, when $\reg^R(k)=0$, these results hold for all integers $i\ge0$.  This is known:  \eqref{eq:kb} 
is proved by Kempf \cite[Lemma 4]{Ke} and is a special case of a result of Backelin \cite[Corollary]{Ba}; see also 
\cite[4.1(a) and 1.4]{ACI}.  On the other hand, \eqref{eq:aci12} is a special case of \cite[4.1(a) and (3.8.3)]{ACI}.

The proof of Theorem \ref{thm:aci1} does not use the material developed in previous sections; rather, it adapts ideas from our proof of \cite[3.1]{ACI}.  We start by 
presenting the relevant background material.

  \begin{bfchunk}{Minimal models.}
Given a set $X=\bigsqcup_{j\ges i\ges1}X_{i,j}$, let $k[X]$ denote the $k$-algebra $\bigotimes_{i=1}^{\infty}k[X_i]$, where $X_i=\bigsqcup_jX_{i,j}$ and $k[X_{i}]$ is the 
exterior algebra on $X_{i}$ when $i$ is odd, respectively, the symmetric algebra on $X_{i}$ when $i$ is even.  

For $x\in X_{i,j}$ set $|x|=i$ and $\deg(x)=j$.  A $k$-basis of $k[X]$ is given by the element $1\in k$ and the monomials $x_1^{d_1}\cdots x_u^{d_u}$ with $x_l\in X$ and 
$d_l=1$ when $|x_{l}|$ is odd, respectively, with $d_l\ge1$ when $|x_{l}|$ is even. We bigrade $k[X]$ by setting
  \[
\big|x_1^{d_1}\cdots x_u^{d_u}\big|=\sum_{l=1}^u d_l|x_l|
\quad\text{and}\quad
\deg\big(x_1^{d_1}\cdots x_u^{d_u}\big)=\sum_{l=1}^u d_l\deg(x_l)\,.
  \]
We also bigrade $S$ by $|y|=0$ and $\deg(y)=i$ for $y\in S_i$.

A \emph{model} of $R$ is a differential bigraded $S$-algebra $S[X]$ with underlying algebra $S\otimes_kk[X]$ and differential satisfying $|\dd(y)|=|y|-1$,
$\deg(\dd(y))=\deg(y)$, and
  \[
\dd(s\otimes yy')=s\otimes \dd(y)y'+(-1)^{|y|}s\otimes y\dd(y')\,,
  \]
together with an $S$-linear map $\varepsilon$ that makes the following sequence exact:
  \begin{equation}
    \label{eq:model1}
\cdots \to S[X]_{i,*}\xra{\,\dd_i\,} S[X]_{i-1,*}\to\cdots\to S[X]_{0,*}\xra{\,\varepsilon\,}R\to0
  \end{equation}

A model $S[X]$ is said to be \emph{minimal} if it satisfies $\dd(S[X])\subseteq (S_{+}+X^2)S[X]$.
Such a model always exists, and any two are isomorphic as differential 
bigraded $S$-algebras; see \cite[7.2.4]{Av:barca}.  In every minimal model $\dd(X_1)$ 
is a minimal set of generators of $J$ and $S[X_1]$ is the Koszul 
complex of the restriction $SX_1\to S$ of $\dd_1$.

Let $X_{0,0}$ be a $k$-basis of $S_1$, and define a set $X'=\bigsqcup_{j\ges i\ges0}X'_{i,j}$ 
by the formula
  \begin{equation}
    \label{eq:model2}
X'_{i,j}=\{x'\}_{x\in X_{i-1,j}}\,.
  \end{equation}
Set $k\langle X'\rangle=\bigotimes_{i=1}^{\infty}k\langle X'_i\rangle$, where $X'_i=\bigsqcup_jX_{i,j}$ and $k\langle X'_i\rangle$ denotes the exterior algebra on $X'_i$ when $i$ is odd, respectively, the divided powers algebra on that space when $i$ is even.  

A $k$-basis of $k\langle X'\rangle$ is given by $1\in k$ and ${x'}_1^{(d_1)}\cdots {x'}_u^{(d_u)}$ with $x'_l\in X'$ and $d_l=1$ if $|x'_{l}|$ is odd, respectively, with $d_l\ge1$ if $|x'_{l}|$ is even. It is bigraded by setting
  \[
\big|{x'}_1^{(d_1)}\cdots {x'}_u^{(d_u)}\big|=\sum_{l=1}^u d_l|x'_l|
\quad\text{and}\quad
\deg\big({x'}_1^{(d_1)}\cdots {x'}_u^{(d_u)}\big)=\sum_{l=1}^u d_l\deg(x'_l)\,.
  \]

By \cite[7.2.6]{Av:barca}, there exists an isomorphism of bigraded $k$-vector spaces
  \begin{equation}
    \label{eq:model3}
\Tor{}Rkk\cong\bigotimes_{i=1}^{\infty}k\langle X'_i\rangle
  \end{equation}
    \end{bfchunk}

  \stepcounter{theorem}
\begin{proof}[Proof of Theorem \emph{\ref{thm:aci1}}]
Let $S[X]$ be a minimal model of $R$.  Set \text{$k[X]=k\otimes_SS[X]$} and note that
\eqref{eq:koszul0} and \eqref{eq:model1} yield isomorphisms of bigraded $k$-algebras 
  \begin{equation}
    \label{eq:aci13}
\hh{k[X]}\cong\Tor{}SkR\cong\HH{}{K}\,.
  \end{equation}
This induces an isomorphism $kX_{1,2}\cong \HH 1K$, by the minimality of $S[X]$.

Let $x'$ be the element of $X'_{i,j}$ corresponding to the element $x$ of $X_{i-1,j}$; see~\eqref{eq:model2}.  For $i+1\le n+1$, from $\preg{n+1}Rk=0$ and \eqref{eq:model3} we get $j=i+1$.  Thus, $|x|\le n$ implies $\deg(x)=|x|+1$. Since $|x|\ge1$ holds for $i\le n$, for any monomial $x_1^{d_1}\cdots x_u^{d_u}$ with $|x_1^{d_1}\cdots x_u^{d_u}|\le n$ we obtain the relations
  \begin{align*}
\deg\big(x_1^{d_1}\cdots x_u^{d_u}\big)
&=\sum_{l=1}^ud_l\deg(x_l)
=\sum_{l=1}^ud_l(|x_l|+1)
=\big|x_1^{d_1}\cdots x_u^{d_u}\big|+\sum_{l=1}^ud_l \\
&\le2\big|x_1^{d_1}\cdots x_u^{d_u}\big|
  \end{align*}

For $j>2i$ these relations yield $(k[X])_{i,j}=0$.  As a consequence, $\HH i{k[X]}_{j}=0$ holds for $j>2i$ and 
$\HH i{k[X]}_{2i}$ is a quotient of $(k[X])_{i,2i}$.  In view of \eqref{eq:aci13}, now \eqref{eq:aci11}
is proved, and for \eqref{eq:aci12} it remains to show $(k[X])_{i,2i}=(kX_{1,2})^i$.

When $\deg\big(x_1^{d_1}\cdots x_u^{d_u}\big)=2\big|x_1^{d_1}\cdots x_u^{d_u}\big|$ the inequality in 
the last display becomes an equality.  All $d_l$ and $|x_l|$ being positive integers, for $1\le l\le u$ we get 
first $|x_l|=1$, then $|x_l|=\deg(x_l)-|x_l|$; that is, $\deg(x_l)=2$.  As a consequence, we get
$k[X]_{i,2i}=k[X_1]_{i,2i}=(kX_{1,2})^i$, as desired.
  \end{proof}

 \section{Subadditivity of syzygies}
 \label{applications}

We keep the hypotheses on $R$ and $S$ from Section \ref{mainthm}.  In this section we pick
up a theme from the introduction---\emph{linearity} at the start of the resolution of $k$ over 
$R$ implies \emph{subadditivity} of the degree of the syzygies of $R$ over~$S$.  We set
 \[
t_i(R)=t^S_i(R)\,,
\quad
\reg(R)=\reg^S(R) \,,
\quad\text{and}\quad
\reg_i(R)=\reg_i^S(R)
 \]
for $i\in\BZ$.  Some of our results depend on the number
  \[
m(R)=\min\{i\in\BZ\mid t_{i}(R)\ge t_{i+1}(R)\}\,.
  \]

  \begin{lemma}
    \label{lem:m(R)}
The following inequalities hold: 
  \begin{alignat}{2}
    \label{eq:m(R)1}
\rank_kR_1 - \dim R&\le m(R)\le \pd_SR\,. 
  \end{alignat}
    \end{lemma}

  \begin{proof}
We set $m=m(R)$ and let $F$ be a minimal free resolution of $R$ over $S$.  

For $h=\pd_SR$ we have $t_h(R)>0$ and $t_{h+1}(R)=-\infty$, hence $m\le h$. 

For $t=t_{m}(R)$ we have $t_{m-1}(R)<t\geq t_{m+1}(R)$.  The functor $?^*=\Hom S?S$
gives homogeneous $R$-linear maps of graded $R$-modules, and hence $k$-linear maps 
  \[
(F_{m-1})^*_{-t}\xra{\,(\dd_{m})^*_{-t}\,} (F_m)^*_{-t}\xra{\,(\dd_{m+1})^*_{-t}\,} (F_{m+1})^*_{-t}
  \]
The inequalities for $t$ imply $(F_{m-1})^*_{-t}=0$ and $(F_{m+1})^*_{j}=0$ for $j<-t$, hence
  \[
\image((\dd_{m+1})^*_{-t})\subseteq(S_+(F_{m+1})^*)_{-t}=\sum_{i\ges1}S_i(F_{m+1})^*_{-t-i}=0
  \]
Thus, we get $\Ext mSRS^t=(F_m)^*_{-t}\ne0$, and hence $\grade_SR\le m$.  This is the desired  
lower bound, as $\grade_SR=\rank_kR_1 - \dim R$ because $S$ is regular. 
  \end{proof}

\begin{theorem} 
  \label{subadKoszul}
Let $a,b$ be non-negative integers such that $\binom{a+b}a$ is invertible in $k$.

When $\preg{a+b+1}Rk=0$ and $\max\{a,b\}\le m(R)$ hold there is an inequality
  \begin{align}
    \label{eq:subadKoszul2}
t_{a+b}(R)
&\leq \max\left\{ 
\begin{aligned}
&t_a(R)+ t_b(R) 
\\
&t_{a-1}(R)+ t_{b-1}(R)+3
\end{aligned}
\right\}\,. 
   \end{align}
In particular, there are inequalities
\begin{alignat}{2}
   \label{eq:subadKoszul1}
t_{a+1}(R)&\leq t_{a}(R)+2
&&\quad \text{when}\quad b=1\,.
 \\
 \label{eq:subadKoszul3}
t_{a+b}(R)
&\leq t_{a}(R)+ t_{b}(R)+1
& &\quad \text{for}\quad b\ge2\,.
   \end{alignat}
   \end{theorem} 

   \begin{proof} 
Formula \eqref{eq:subad} and the hypothesis on $a,b$ yield
  \[
t_{a+b}(R)
\leq \max\left\{ 
\begin{aligned}
&t_a(R) + t_b(R)
\\
&t_{a-1}(R)-(a-1)+t_{b-1}(R)-(b-1)+a+b+1
\end{aligned}
\right\} 
  \]
This  simplifies to \eqref{eq:subadKoszul2}. The other inequalities follow, since $t_{0}(R)=0$, $t_{1}(R)=2$, and $t_{i-1}\leq t_{i}-1$ for $i\leq m(R)$.
 \end{proof} 

Unless $R$ is a polynomial ring, $\preg{2}Rk=0$ implies $t_1(R)=2$.  Thus, \eqref{eq:kb} 
gives an upper bound on $t_a(R)$ by a function depending only on $a$ and $t_1(R)$.   
Such a property does not hold in general:

\begin{example} 
For $R=k[x_1,\dots,x_e]/J$ and $J$ generated by $e+1$ general quadrics
  \[
t_1(R)=2
\quad\text{and}\quad
t_2(R)=\left\lfloor\frac e2\right\rfloor +2\,.
  \]
This is seen from the free resolution given by Migliore and Mir\'o-Roig in \cite[5.4]{MiMi}. 
  \end{example}  

On the other hand, we have:
  
  \begin{bfchunk}{General upper bounds.} 
    \label{ch:exponential}
For every algebra $R$ the number $t_a(R)$ is bounded above by a function of the numbers 
$a$, $t_1(R)$, \emph{and} $e=\rank_kR_1$; specifically, Bayer and Mumford \cite[3.7]{BM} 
and Caviglia and Sbarra \cite[2.7]{CS} give the upper bound
\begin{equation*}
t_a(R)\leq (2t_1(R))^{2^{e-2}} +a -1
  \quad\text{for every}\quad a\ge1\,.
\end{equation*}
It cannot be tightened substantially, since variations of the Mayr-Meyer ideals \cite{MM} 
yield algebras with the same type of doubly exponential syzygy growth, albeit with different coefficients; 
see Bayer and Stillman \cite[2.6]{BS} and Koh \cite[Theorem, p.\,233]{Ko}. 
  \end{bfchunk}

The results in this section suggest the following 

  \begin{conjecture}
    \label{con:conjecture}
If $\preg{a+b+1}Rk=0$, then the following inequality holds:
  \begin{equation}
    \label{eq:question}
t_{a+b}(R)\leq t_{a}(R)+t_{b}(R)
\quad\text{whenever}\quad 
a+b\le \pd_SR\,.
  \end{equation}
     \end{conjecture}
   
Additional supporting evidence is reviewed below. The algebra $R=S/J$ has \emph{monomial 
relations} if $J$ can be generated by products of elements in some $k$-basis of~$R_1$; 
such an algebra is Koszul if and only if $J$ is generated by quadrics; see  \cite{F}.

  \begin{examples} 
    \label{ex:conjecture}
Set $e=\rank_k R_1$ and $h=\pd_SR$.

The inequality \eqref{eq:question} holds in the following cases:
  \begin{enumerate}[\rm(1)]
     \item 
$\dim R\le1$, $\depth R=0$, and $a+b=e$: 
Eisenbud, Huneke, and Ulrich \cite[4.1]{EHU}.
    \item 
$b=1$ when $R$ has monomial relations:  Fern\'andez-Ramos and Gimenez \cite[1.9]{FG}
in the quadratic case, Herzog and Srinivasan \cite{HS} in general.
    \item 
$e\le7$ when $R$ has quadratic monomial relations and $\chr(k)=0$:  Verified by 
McCullough by using \textsc{Macaulay2} \cite{M2} and \textsc{nauty}~\cite{Na}.
  \end{enumerate}

In addition, the following weaker inequalities are known to hold:
  \begin{enumerate}[\rm(1)]
 \item [\rm(4)]
$t_{a+b}(R)\le t_a(R)+t_b(R)+1$ when $R$ is Cohen-Macaulay, $\preg{h+1}Rk=0$, and $\binom{a+b}a$ is 
invertible in $k$:  Formulas \eqref{eq:m(R)1} and \eqref{eq:subadKoszul3}. 
 \item [\rm(5)]
$t_{a+b}(R)\le\max_{a,b\ges1} \{t_a(R)+t_b(R)\}$ when $a+b=e$: McCullough \cite[4.4]{Mc}.
  \end{enumerate}
   \end{examples} 

When $p$ is a prime number and $n$ a non-negative integer, we say that $p$ is 
\emph{good for} $n$ if $\binom ni\not\equiv 0\pmod p$ holds for all $i\in\BZ$ with $0\le i\le n$.  

\begin{lemma}  
  \label{lem:good}
A prime number $p$ is good for an integer $n\ge0$ if and only if one of the following holds: 
\emph{(a)} $p=0$, or \emph{(b)} $p>n$, or \emph{(c)} $n+1=p^su$ for some integers $s\ge1$ 
and $2\le u\le p$; furthermore, at most one pair $(p,n)$ satisfies \emph{(c)}.
\end{lemma}

  \begin{proof} 
We may assume $0< p\le n$.  Let $\sum_{j=0}^sn_jp^j$ and $\sum_{j=0}^ri_jp^j$ be the $p$-adic 
expansions of $n$ and $i$, respectively; thus, $s\ge1$ and $1\le n_s< p$.  Lucas' congruence
  \[
\binom ni\equiv\prod_{j=0}^s\binom{n_j}{i_j}\pmod p
  \]
shows that $\binom ni$ is invertible mod $p$ if and only if so is $\binom{n_j}{i_j}$ for each $j$. One has $0\le i_j,n_j<p$ by definition, so $\binom{n_j}{i_j}$ is invertible mod $p$ if and only if $n_j\ge i_j$. In conclusion, $p$ is good for $n$ if and only if $n_j\ge i_j$ holds for $j=0,\dots,s$ and $n\ge i\ge0$.  

An inequality $n\ge i\ge0$ means $s\ge r$ and there is an integer $l$ with $0\ge l\ge s$, satisfying 
$n_{j}=i_{j}$ for $j>l$ and $n_{l}>i_{l}$.  This holds for all $i$ with $n\ge i\ge0$ if and only if 
$n_{s}\ge i_{s}$ and $n_j=p-1$ for $s> j\ge0$; that is, if and only if 
  \[
n=n_sp^s+\sum_{j=0}^{s-1}(p-1)p^j=n_sp^s+(p-1)\frac{p^s-1}{p-1}=(n_s+1)p^s-1\,.
  \]

Assume $up^s=u'p'^{s'}$ holds for some prime number $p'$ and integers $u',s'$ satisfying $s'\ge1$ and $2\le u'\le p$ .  
If $p'>p$, then $p'|u$, and hence $u\ge p'>p$, which is impossible.  By symmetry, so is $p>p'$, so we
get $p'=p$.  Now $s\ne s'$ implies $p|u$ or $p|u'$, which again is impossible.  Thus, we must have 
$s=s'$, and hence $u=u'$. 
  \end{proof} 

Subadditivity is also reflected in the behavior of partial regularity. Recall from  \eqref{eq:m(R)2} that one has always an inequality $\preg {a}{}R\leq \preg {a+1}{}R$.

\begin{proposition} 
\label{subadKoszul2}
Let $a,b$ be non-negative integers.

If $\preg{a+b+1}Rk=0$ and $p$ is good for $(a+b)$, then the following inequalities hold:
  \begin{alignat}{2}
    \label{eq:subadKoszul11}
\preg{a+1}{}R
&\leq\preg{a}{}R+1
&\quad\text{when}\quad b&=1\,,
 \\
    \label{eq:subadKoszul21}
\preg{a+b}{}R
&\leq \max\left\{ 
\begin{aligned}
&\preg{a}{}R+\preg{b}{}R
\\
&\preg{a-1}{}R+\preg{b-1}{}R+1
\end{aligned}
\right\} 
&\quad\text{for}\quad b&\ge2\,.
   \end{alignat}
     \end{proposition}

   \begin{proof} 
Pick any integer $i$ with $0\le i\le a+b$ and write it as $i=a'+b'$ with integer $a'$ and $b'$ satisfying $0\le a'\le a$ 
and $0\le b'\le b$.  The number $\binom{a'+b'}{b'}$ is invertible in $k$ due to the hypothesis on $\chr(k)$, so 
for $b\ge1$ formulas \eqref{eq:subad} and \eqref{eq:m(R)2} yield:
   \begin{align*}
t_{i}(R)-i 
&\leq \max\left\{ 
\begin{aligned}
&(t_{a'}(R)-a')+ (t_{b'}(R)-b')
  \\
&\preg{a'-1}{}R+\preg{b'-1}{}R+1
\end{aligned}
\right\} 
  \\
&\leq \max\left\{ 
\begin{aligned}
&\preg{a'}{}R+\preg{b'}{}R
  \\
&\preg{a'-1}{}R+\preg{b'-1}{}R+1
\end{aligned}
\right\} 
  \\
&\leq \max\left\{ 
\begin{aligned}
&\preg{a}{}R+\preg{b}{}R
  \\
&\preg{a-1}{}R+\preg{b-1}{}R+1
\end{aligned}
\right\} 
 \end{align*}
This implies \eqref{eq:subadKoszul21}, and also \eqref{eq:subadKoszul11} because $\preg j{}R=j$ holds for $j=0,1$. 
  \end{proof}

 \section{The Green-Lazarsfeld property $N_q$}
   \label{applications2}

We keep the hypotheses on $R$ and $S$ from Section \ref{mainthm}, and set 
  \[
p=\chr(k)\,,
\quad
h=\pd_RS\,, 
\quad
e=\rank_k R_1\,.
  \]

Green and Lazarsfeld \cite[p.\,85]{GL} say that $R$ has property $N_q$ for a positive integer $q$ 
if $t_i(R)=i+1$ holds for $1\le i\le q$.  This is equivalent to $\preg qSR=1$, see \ref{ch:regs}.   

  \begin{theorem} 
\label{lowreg}
Let $R$ be a graded algebra satisfying $\preg{n+1}Rk=0$ for some $n\ge2$.

When $R$ satisfies $N_q$ with $q\geq 2$  the following inequalities hold:
  \begin{align}
    \label{eq:lowreg0}
t_i(R)&\leq 2i-1 \quad\text{for}\quad 2\le i\le n\,.
\intertext{\indent If, furthermore, $p$ is good for $i$, then also the following inequalities hold:}
    \label{eq:lowreg1}
t_i(R)&\leq 
2\left\lfloor \frac{i}{q+1}\right\rfloor+i+
\begin{cases}
0 & \text{if}\quad (q+1)|i
  \\
1& \text{otherwise. }
\end{cases}
  \\
    \label{eq:lowreg2}
\preg n{}R &\leq  2\left\lfloor \frac{n}{q+1}\right\rfloor+1
  \end{align}
  \end{theorem} 
  
  \begin{Remark} 
Both \eqref{eq:lowreg0} and \eqref{eq:lowreg1}  sharpen \eqref{eq:kb}.  The condition $q\ge2$ is necessary, 
since a complete intersection of quadrics $R$ has $N_{1}$ and $t_a(R)=2a$ for $a\le h$.
  \end{Remark} 

  \begin{proof} 
The isomorphisms \eqref{eq:koszul1} and \eqref{eq:aci12} give for each $i\ge2$ equalities
  \[
\Tor{i}SRk_{2i}=(\Tor{1}SRk_{2})^2(\Tor{1}SRk_{2})^{i-2}=\Tor{2}SRk_{4}(\Tor{1}SRk_{2})^{i-2} 
  \]
We have $\Tor{2}SRk_4=0$ because $R$ satisfies $N_2$, and hence $\Tor{i}SRk_{2i}=0$.  Since
\eqref{eq:aci13} gives $t_{2i}\le 2i$ for $i\ge1$, we conclude that $t_{2i}\le 2i-1$ holds for $i\ge2$.

If $i\le q$, then \eqref{eq:lowreg1} holds by hypothesis.  When $i\geq q+1$
we may assume by induction on $i$ that for $0\le j<i$ the following inequalities hold:
  \begin{equation}
    \label{eq:lowreg3}
\preg j{}R\leq  2\left\lfloor \frac{j}{q+1}\right\rfloor+
  \begin{cases}
0 & \text{if}\quad (q+1)|j
  \\
1& \text{otherwise.}
\end{cases}
  \end{equation}
Euclidean division yields integers $0\le u\le q$ and $v\ge1$, such that 
  \[
i=(q+1)v+u\,.
  \]

In case $u=0$ by applying first formula \eqref{eq:subad} with $a=i-1$ and $b=1$, 
then formula \eqref{eq:lowreg3} with $j=(q+1)v-1$, $j=(q+1)v-2$, and $j=0$ we obtain
 \begin{align*}
t_i(R) 
& \leq \max\left\{
\begin{aligned}
&t_{(q+1)v-1}(R)+t_1(R)
  \\
&\preg{(q+1)v-2}{}R+\preg{0}{}R+i+1 
 \end{aligned}
 \right\} 
   \\
& \leq \max\left\{
\begin{aligned}
&2(v-1)+(i-1)+1+2
  \\
&2(v-1)+1+0+i+1 
 \end{aligned}
 \right\} 
 \\
&=2v+i
  \end{align*}

In case $1\le u\le q$ by applying formula \eqref{eq:subad} with $a=(q+1)v$ and $b=u$, 
then formula \eqref{eq:lowreg3} with $j=(q+1)v$, $j=u$, $j=(q+1)v-1$, and $j=u-1$ we get
 \begin{align*}
t_i(R) 
& \leq \max\left\{
\begin{aligned}
&t_{(q+1)v}(R)+t_u(R) 
  \\
&\preg{(q+1)v-1}{}R+\preg{u-1}{}R+i+1
 \end{aligned}
 \right\} 
   \\
& \leq \max\left\{
\begin{aligned}
&2v+(q+1)v+u+1
  \\
&2(v-1)+1+1+i+1 
 \end{aligned}
 \right\} 
 \\
&=2v+i+1.
  \end{align*}

This concludes the proof of \eqref{eq:lowreg1}. Formula \eqref{eq:lowreg2} 
is a direct consequence. 
  \end{proof}

Conjecture \ref{con:conjecture} predicts that the bounds in Theorem \ref{lowreg} can be strengthened:

  \begin{lemma}
Assume $R$ has property $N_q$ for some $q\ge2$.

If $t_{i+q}(R)\leq t_{i}(R)+t_{q}(R)$ holds for $i\le h-q$, then there are inequalities
  \begin{align}
    \label{eq:conjecture1}
t_i(R)&\leq 
\left\lceil\frac{i}{q}\right\rceil+i
\quad\text{for}\quad 
i\ge1\,.
  \\
    \label{eq:conjecture2}
\reg R &\leq  \left\lceil \frac{h}{q}\right\rceil\,.
  \end{align}
    \end{lemma}

  \begin{proof}
For $1\le i\le q$ the inequality \eqref{eq:conjecture1} is just condition $N_q$.  Thus, by induction 
on $i$ we may assume that the inequality holds for some $i\ge q$.  The calculation
  \begin{align*}
t_{i+1}(R)
&\le t_{i+1-q}(R)+t_{q}(R) \\
&\le\left\lceil\frac{i+1-q}{q}\right\rceil+i+1-q+q+1 \\
&=\left\lceil\frac{i+1}{q}\right\rceil+i+1
  \end{align*}
completes the induction step for formula \eqref{eq:conjecture1}.  The latter implies \eqref{eq:conjecture2}.
  \end{proof}

Next we describe graphical displays that help visualize such inequalities.

  \begin{bfchunk}{Betti templates.}
\label{ch:tables}
Recall that the \emph{Betti table} of $R$ is the matrix that has $\beta_{i,i+j}(R)$ as the entry in its 
$i$th column and $j$th row.  In order to present information on the vanishing of Betti numbers, for 
each integer $q\ge1$ we form a \emph{Betti template} by replacing $\beta_{i,i+j}(R)$ 
with one of the symbols from the following list: 
  \begin{enumerate}[\rm\quad(1)]
\item[\st] non-zero due to condition $N_q$.
\item[-]  zero due to condition $N_q$.
\item[\bu]  zero by \eqref{eq:lowreg0}. 
\item[\ze] zero by \eqref{eq:lowreg1} when $p$ is good for $i$, predicted by \eqref{eq:conjecture1} in general.
\item[\bc] zero predicted by \eqref{eq:conjecture1}.
\item[\et] no information available.
     \end{enumerate} 
Thus, columns $0$ through $i$ of the template contain the available and the conjectural information 
on the vanishing of $\beta_{i,i+j}(R)$ when $\preg{i+1}Rk=0$ and $R$ satisfies $N_q$.

An additional row at the bottom of the template displays---if it exists---the only good 
characteristic $p$ for $i$ other that $p=0$ or $p>i$; see Lemma \ref{lem:good}.
  \end{bfchunk}

A natural question is how tightly these templates can be filled in.  Precise answers would be interesting both in 
their own right and for applications to coordinate rings of classical algebraic varieties.  We present estimates by 
drawing on existing  examples, and also record new information obtained from our results.  It is worth noting 
that the algebras in these examples are Koszul and many are Gorenstein. 

Recall that the \emph{Hilbert series} $H_R(t)=\sum_{i\ges0}\rank_kR_it^i$ is a rational function with
numerator in $\BZ[t]$ and denominator equal to $(1-t)^{\dim R}$.

    \begin{bfchunk}{Minimally elliptic singularities.} 
      \label{wahl} 
Let $k=\mathbb{C}$ and let $R$ be the associated graded ring of a rationally elliptic surface singularity; 
see Laufer \cite{La}.  Wahl \cite[2.8]{Wa} proved that $R$ is a Gorenstein ring, $\pd_S(R)=e-2\ge2$, 
$\beta_{i ,j}(R)=0$ for $j\ne i+1$ and $1\le i\le e-3$, $\beta_{e-2 ,j}(R)=0$ for $j\ne e$, and $H_R(t)=(1+et+t^2)/(1-t)^2$.  

When $k$ is infinite Wahl's arguments apply to any Gorenstein $k$-algebra that has $H_R(t)=(1+et+t^2)/(1-t)^{\dim R}$.  
Furthermore, such an $R$ is Koszul:  By factoring out a maximal $R$-regular sequence of linear forms it suffices to treat 
the case of dimension zero, where the desired assertion follows from \cite[Theorem 3(1)]{LA}.

In particular, for every $q\ge2$ there is a Koszul $k$-algebra $R$ that is a Gorenstein ring, 
and for which equality holds in \eqref{eq:conjecture1} with $i=q+1$.
 \end{bfchunk} 

  \begin{example} 
\label{1551}
The first $13$ columns and $8$ rows of the Betti template for $N_2$:
  \begin{align*}
&\text{
\begin{tabular}{c|ccccccccccccccccc}
  &$\scriptstyle  0$ &$\scriptstyle  1$ & $\scriptstyle  2$ & $\scriptstyle  3$& $\scriptstyle  4$ & $\scriptstyle  5$& $\scriptstyle  6$ & $\scriptstyle  7$ & $\scriptstyle  8$ & $\scriptstyle  9$ & $\scriptstyle  10$ & $\scriptstyle  11$ & $\scriptstyle  12$ & $\scriptstyle  13$ \\
\hline
 $\scriptstyle  0$ &\st & - & - & - & - & - & - & - & - & - & - & - & - & -  \\ \cline{2-2}
 $\scriptstyle  1$ &\multicolumn{1}{c|}{-} &\st&\st&\et&\et&\et&\et&\et&\et&\et&\et&\et&\et&\et  \\  \cline{3-4}
 $\scriptstyle  2$ & - & - &\multicolumn{1}{c|}{-}&\et&\et&\et&\et&\et&\et&\et&\et&\et&\et&\et  \\  \cline{5-6}
 $\scriptstyle  3$ & - & - & - & \multicolumn{1}{c|}{\bu}&\multicolumn{1}{c|}{\bc}&\et&\et&\et&\et&\et&\et&\et&\et&\et \\  \cline{6-8}
 $\scriptstyle  4$ & - & - & - & \ci & \bu & \multicolumn{1}{c|}{\ze} &\multicolumn{1}{c|}{\bc}&\et&\et&\et&\et&\et&\et&\et   \\  \cline{8-10}
 $\scriptstyle  5$ & - & - & - & \ci & \ci & \bu & \multicolumn{1}{c|}{\ze} &\bc&\multicolumn{1}{c|}{\bc}&\et&\et&\et&\et&\et  \\ \cline{9-12}
 $\scriptstyle  6$ & - & - & - & \ci & \ci & \ci & \bu &\ze &\multicolumn{1}{c|}{\ze}&\bc&\multicolumn{1}{c|}{\bc}&\et&\et&\et   \\  \cline{11-11} \cline{13-14}
 $\scriptstyle  7$ & - & - & - & \ci & \ci & \ci & \ci &\bu &\ze &\multicolumn{1}{c|}{\ze}&\bc&\bc&\multicolumn{1}{c|}{\bc} &\et  \\  \cline{12-13} \cline{15-15}
 $\scriptstyle  8$ & - & - & - & \ci & \ci & \ci & \ci &\ci &\bu &\ze &\ze &\multicolumn{1}{c|}{\ze} &\bc &\bc \\  \cline{14-15}
 \hline\hline
 $\scriptstyle  p$ &&&&&                              &$\scriptstyle 3$ &$\scriptstyle  $ &$\scriptstyle 2$   &$\scriptstyle 3$   &$\scriptstyle 5$   &$\scriptstyle $   &$\scriptstyle $   &$\scriptstyle $   &$\scriptstyle 7$   
\end{tabular}
}
\intertext{\indent
Let $R$ be a general Gorenstein algebra with Hilbert series $1+5z+5z^2+z^3$. 
It is Koszul by \cite[6.3]{CRV}.  When $p\ne2$, by specializing Schreyer's result \cite[p.\,107]{S} 
to the Artinian case one sees that the Betti table of $R$ is given by}
&\text{
 \begin{tabular}{c|cccccccccccccc}
 &$\scriptstyle  0$ &$\scriptstyle  1$ & $\scriptstyle  2$ & $\scriptstyle  3$& $\scriptstyle  4$ & $\scriptstyle  5$ & $\scriptstyle  6$  \\
\hline
 $\scriptstyle  0$ & $\scriptstyle\textbf{1}$ & -   &    -  & -  & -   & - & -    \\
 $\scriptstyle  1$ & - & $\scriptstyle\textbf{10}$ & $\scriptstyle\textbf{16}$  & -  & -   & - & -      \\  
 $\scriptstyle  2$ & - &   -  & -    &$\scriptstyle\textbf{16}$ &$\scriptstyle\textbf{10}$ & - & -       \\  
 $\scriptstyle  3$ & - & -    & -    & -  & -    &$\scriptstyle\textbf{1}$ & -    \\  
 $\scriptstyle  4$ & - & -    & -    & -  & -    & - & -    
  \end{tabular}
  }
  \end{align*}

Thus, when $q=2$ there is a Koszul $k$-algebra $R$ that is a Gorenstein ring, and for 
which equalities hold in \eqref{eq:conjecture1} with $i=3$ and $i=5$. 
   \end{example} 

Property $N_3$ appears prominently in recent studies of a basic construction:

  \begin{bfchunk}{Segre products.}
    \label{ch:segre}
Let $(e_1,\dots,e_s)$ be an $s$-tuple of positive integers with $s\ge2$.

Let $R^{(e_1,\dots,e_s)}$ be the coordinate ring of the Segre embedding 
${\mathbb P}^{e_1}_k\times\cdots\times{\mathbb P}^{e_s}_k\to{\mathbb P}^{e}_k$, where $e=\prod_{i=1}^s(e_i+1)-1$.
It is Koszul; see B\u{a}rc\u{a}nescu and Manolache \cite[2.1]{BMa}.  

When $k=\mathbb C$ Rubei \cite[Theorem 10]{Ru} proved that $R^{(e_1,\dots,e_s)}$ satisfies $N_3$, but not $N_4$ 
unless $s=2$ and $e_1=1$ or $e_2=1$.  Snowden developed a functorial method for building its syzygies:  
``One consequence of this result is that almost any question about $p$-syzygies---such as, 
is the module of $13$-syzygies of every Segre embedding generated in degrees at most $20$?---can be 
resolved by a finite computation (in theory);'' see \cite[p.\,226]{Sn}. Column $13$ in the next example gives 
a positive answer.  
  \end{bfchunk}

    \begin{example} 
      \label{Rubei} 
The first $16$ columns and $8$ rows of the Betti template for $N_3$:
\begin{align*}
&\text{
\begin{tabular}{c|ccccccccccccccccc}
  &$\scriptstyle  0$ &$\scriptstyle  1$ & $\scriptstyle  2$ & $\scriptstyle  3$& $\scriptstyle  4$ & $\scriptstyle  5$& $\scriptstyle  6$ & $\scriptstyle  7$ & $\scriptstyle  8$ & $\scriptstyle  9$ & $\scriptstyle  10$ & $\scriptstyle  11$ & $\scriptstyle  12$ & $\scriptstyle  13$ & $\scriptstyle  14$ & $\scriptstyle  15$ & $\scriptstyle  16$ \\
\hline
 $\scriptstyle  0$ &\st& - & - & - & - & - & - & - &  - & - & - & - & - & - & - & - & -  \\ \cline{2-2}
 $\scriptstyle  1$ &\multicolumn{1}{c|}{-} &\st&\st&\st&\et&\et &\et&\et&\et&\et&\et&\et&\et&\et&\et&\et&\et \\  \cline{3-5}
 $\scriptstyle  2$ & - & - & - &\multicolumn{1}{c|}{-} &\et&\et &\et&\et&\et&\et&\et&\et&\et&\et&\et&\et&\et \\  \cline{6-8}
 $\scriptstyle  3$ & - & - & - & - & \multicolumn{1}{c|}{\ze}&\bc &\multicolumn{1}{c|}{\bc} & \et&\et&\et&\et&\et&\et&\et&\et&\et&\et \\  \cline{7-11}
 $\scriptstyle  4$ & - & - & - & - & \bu &\ze&\ze& \multicolumn{1}{c|}{\ze} & \bc&\multicolumn{1}{c|}{\bc}&\et&\et&\et&\et&\et&\et&\et  \\ \cline{10-10} \cline{12-14}
 $\scriptstyle  5$ & - & - & - & - & \ci &\bu &\ze&\ze &\multicolumn{1}{c|}{\ze} & \bc & \bc &\bc&\multicolumn{1}{c|}{\bc}&\et&\et&\et&\et  \\ \cline{11-13} \cline{15-17}
 $\scriptstyle  6$ & - & - & - & - & \ci &\ci &\bu &\ze & \ze&\ze &\ze & \multicolumn{1}{c|}{\ze} &\bc&\bc&\bc&\multicolumn{1}{c|}{\bc}&\et \\  \cline{14-14} \cline{18-18}
 $\scriptstyle  7$ & - & - & - & - & \ci &\ci &\ci &\bu & \ze&\ze &\ze&\ze & \multicolumn{1}{c|}{\ze} &\bc&\bc&\bc&\bc \\  \cline{15-17} 
 $\scriptstyle  8$ & - & - & - & - & \ci &\ci &\ci &\ci & \bu&\ze &\ze &\ze &\ze &\ze &\ze & \multicolumn{1}{c|}{\ze} &\bc \\  \cline{18-18} 
 \hline\hline
 $\scriptstyle  p$ &&&&&                              &$\scriptstyle 3$ &$\scriptstyle  $ &$\scriptstyle 2$   &$\scriptstyle 3$   &$\scriptstyle 5$   &$\scriptstyle $   &$\scriptstyle $   &$\scriptstyle $   &$\scriptstyle 7$   &$\scriptstyle 5$   &$\scriptstyle 2$   &$\scriptstyle $
    \end{tabular}
}
  \end{align*}

The algebra $R^{(1,1,1)}$  has $N_3$ by \ref{ch:segre}, it is Gorenstein with $H_R=(1+4t+t^2)/(1-t)^4$
by \cite[3.2(7)]{BMa}, and so it is Koszul by \ref{wahl}. 
 \end{example} 

Equality can hold in \eqref{eq:conjecture2} for any $q=h-1\ge2$ by \ref{wahl} and with $(q,h)=(2,5)$
by Example \ref{1551}.  However, \eqref{eq:conjecture2} may be far from optimal when $\pd_SR$ is large.  
In fact, we know of no family of Koszul algebras satisfying $N_2$  and with projective dimensions bounded 
above by some linear function of their regularities.  On the other hand, families with quadratic bounds do exist:

  \begin{bfchunk}{Grassmannians.}
    \label{ch:grass}
Let $R^{[n]}$ be the coordinate ring of the Grassmannian of $2$-dimensional subspaces of the affine space 
${\mathbb A}^n_k$.  It is classically known that the $k$-algebra $R^{[n]}$ is minimally generated by $\binom n2$ 
elements, and Hochster \cite[3.2]{Ho} proved that it is a Gorenstein ring of Krull dimension $2n-3$.  It is also 
Koszul, see \cite{Ke}, and when $p=0$ or $p>\frac12(n-4)$ has property $N_2$ by Kurano \cite[5.3]{Ku}.  

The following hold: $\reg(R^{[n]})=n-3$ by \cite[1.4]{BH} but $\pd(R^{[n]})=\binom{n-2}{2}$.
  \end{bfchunk}

In a different direction, Dao, Huneke, and Schweig \cite[4.8]{DHS} obtained an upper bound on the regularity of 
Koszul algebras with monomial relations satisfying $N_q$ for some $q\ge2$ in terms of the number $e$ of algebra 
generators: They exhibit an explicit polynomial $f_q(x)\in \BR[x]$ and a real number $\alpha_q\in \BR$ with 
$\alpha_q>1$, such that
  \[
\reg R\leq \log_{\alpha_q} (f_q(e))\,.
  \]
Our last example shows that no bound of similar type can exist for the class of \textit{all} Koszul algebras 
satisfying~$N_q$, since it contains the following family:

  \begin{bfchunk}{Veronese subalgebras.}
Let $R^{q,n}$ be the $k$-subalgebra generated by all the monomials of degree $q$ in a polynomial ring in 
$qn$ variables.  It is Koszul by \cite[2.1]{BMa}, has $N_q$ by \cite[4.2]{BCR1}, and needs $\binom{qn+q-1}{q}$ 
generators, and its regularity is equal to $(q-1)n$; the last expression is well known:  It follows easily from 
the results of Goto and Watanabe in \cite[Ch.\,3]{GW} and is given explicitly by Nitsche in \cite[4.2]{Ni}.
  \end{bfchunk}

 \end{document}